\documentclass[11pt]{amsart}

\usepackage{amsmath}
\usepackage{amssymb,mathtools,amsthm,latexsym}

\calclayout

\title{Pointwise inequalities for Sobolev functions on outward cuspidal domains}
\author{Sylvester Eriksson-Bique, Pekka Koskela, Jan Mal\'y 
and Zheng Zhu}

\address{Sylvester Eriksson-Bique\\
Department of Mathematics\\
University of California,
Los Angeles, Box 95155, Los Angeles, CA-90095-1555, USA}
\email{\tt syerikss@math.ucla.edu}

\address{Pekka Koskela\\
Department of Mathematics and Statistics\\
University of Jyv\"askyl\"a, P.O. Box 35 (MaD),
FI-40014, Jyv\"askyl\"a, Finland}
\email{\tt pekka.j.koskela@jyu.fi}

\address{Jan Mal\'y\\
Department of Mathematics\\
Faculty of Science,
 J. E. Purkyn\v e University\\
\v Cesk\'e ml\'ade\v{z}e 8\\
400 96 \'Ust\'{\i} nad Labem\\
Czech Republic}
\email{jan.maly1@ujep.cz}

\address{Zheng Zhu\\
Department of Mathematics and Statistics\\
University of Jy\"askyl\"a, P.O. Box 35 (MaD),
FI-40014, Jyv\"askyl\"a, Finland}
\email{\tt zheng.z.zhu@jyu.fi}

\usepackage{stackrel}

\usepackage{color}

\numberwithin{equation}{section}
\long\def\colred#1\endred{{\color{red}#1}}
\long\def\colgreen#1\endgreen{{\color{green}#1}}
\long\def\colmagenta#1\endmagenta{{\color{magenta}#1}}
\long\def\colblue#1\endblue{{\color{blue}#1}}
\long\def\colyellow#1\endyellow{{\color{yellow}#1}}

\usepackage{tikz}
\usetikzlibrary{matrix,arrows}

\theoremstyle{plain}
\newtheorem{theorem}[equation]{Theorem}

\newtheorem{lemma}[equation]{Lemma}

\theoremstyle{remark}

\theoremstyle{definition}

\newtheorem{defn}[equation]{Definition}

\newtheorem*{question*}{Question}

\usepackage{graphicx}


\thanks{The first author was partially supported by the National Science Foundation under grant \#DMS-1704215. The second and fourth authors have been supported  by the Academy of Finland via Centre of Excellence in Analysis and Dynamics Research (Project \#307333). The fourth author was also supported by the China Scholarship Council fellowship (Project \#201506020103). The authors also are thankful for IMPAN for hosting the semester ``Geometry and analysis in function and mapping theory on Euclidean and metric measure space'' where part of this research was conducted. This work was also partially supported by the grant \#346300 for IMPAN from the Simons Foundation and the matching 2015-2019 Polish MNiSW fund.}

\newcounter{prob}
\def\rr{{\mathbb R}}
\def\rn{{{\rr}^n}}

\def\fz{\infty}

\def\boz{{\Omega}}

\def\bint{{\ifinner\rlap{\bf\kern.25em--}
\int\else\rlap{\bf\kern.45em--}\int\fi}\ignorespaces}

\def\bbint{{\ifinner\rlap{\bf\kern.25em--}
\hspace{0.078cm}\int\else\rlap{\bf\kern.45em--}\int\fi}\ignorespaces}

\def\mp{{M^{1,p}(\boz)}}

\def\wp{{W^{1,p}(\boz)}}

\def\lp{{L^{p}(\boz)}}

\def\r{\right}
\def\lf{\left}

\def\setcolon{;}

\newcommand{\R}{\ensuremath{\mathbb{R}}}

\def\XXint#1#2#3{{\setbox0=\hbox{$#1{#2#3}{\int}$ }
\vcenter{\hbox{$#2#3$ }}\kern-.58\wd0}}

\newcommand{\co}{\mskip0.5mu\colon\thinspace}   

\def\vint_#1{\mathchoice%
        {\mathop{\kern 0.2em\vrule width 0.6em height 0.69678ex depth -0.58065ex
                \kern -0.8em \intop}\nolimits_{\kern -0.4em#1}}%
        {\mathop{\kern 0.1em\vrule width 0.5em height 0.69678ex depth -0.60387ex
                \kern -0.6em \intop}\nolimits_{#1}}%
        {\mathop{\kern 0.1em\vrule width 0.5em height 0.69678ex
            depth -0.60387ex
                \kern -0.6em \intop}\nolimits_{#1}}%
        {\mathop{\kern 0.1em\vrule width 0.5em height 0.69678ex depth -0.60387ex
                \kern -0.6em \intop}\nolimits_{#1}}}
\def\vintslides_#1{\mathchoice%
        {\mathop{\kern 0.1em\vrule width 0.5em height 0.697ex depth -0.581ex
                \kern -0.6em \intop}\nolimits_{\kern -0.4em#1}}%
        {\mathop{\kern 0.1em\vrule width 0.3em height 0.697ex depth -0.604ex
                \kern -0.4em \intop}\nolimits_{#1}}%
        {\mathop{\kern 0.1em\vrule width 0.3em height 0.697ex depth -0.604ex
                \kern -0.4em \intop}\nolimits_{#1}}%
        {\mathop{\kern 0.1em\vrule width 0.3em height 0.697ex depth -0.604ex
                \kern -0.4em \intop}\nolimits_{#1}}}

\begin{document}

\maketitle

\begin{abstract}
We show that the first order Sobolev spaces $W^{1,p}(\Omega_\psi),$ $1<p\leq\fz,$ on cuspidal
symmetric domains $\Omega_\psi$ can be characterized via pointwise inequalities. In particular, 
they coincide with the Haj\l{}asz-Sobolev spaces $M^{1,p}(\Omega_\psi)$.\\
\end{abstract}

\section{Introduction}
Optimal definitions for Sobolev spaces are crucial in analysis. It was a remarkable discovery of Haj\l{}asz \cite{Pitor} that 
distributionally defined Sobolev  functions can be characterized 
using pointwise estimates in the context of Sobolev extension 
domains. This, in part, has played a crucial role in defining Sobolev 
spaces for general metric measure spaces. Here, we show that for 
certain cuspidal domains the pointwise characterization holds without 
any additional assumptions. These domains do not admit extensions for 
Sobolev functions.
Given a domain $\Omega\subset \rn,$ we denote by $W^{1,p}(\Omega),$ $1\le p\le
\infty,$ the usual first order Sobolev space consisting of all functions
$u\in L^p(\Omega)$ whose first order distributional partial derivatives also
belong to $L^p(\Omega).$ If $\Omega=\rn,$ then any Sobolev function $u$ 
satisfies the pointwise inequality 
\begin{equation}\label{af}
|u(x)-u(y)|\le C|x-y|\left(
M[ |\nabla u|]
(x)+
M[|\nabla u|]
(y)\right )
\end{equation}
at Lebesgue points of $u,$ where 
$M[|\nabla u|]$ 
is the Hardy-Littlewood maximal function
of $|\nabla u|,$ see \cite{AandF, BandH, Pitor, Lewis}. 
Motivated by this, P. Haj\l asz
introduced in \cite{Pitor} the space $M^{1,p}(\Omega)$ consisting of all 
those $u\in L^p(\Omega)$
for which there exists a set $E\subset \Omega$ of $n$-measure zero and a function 
$
0\leq g
\in L^p(\Omega)$ so that
\begin{equation}\label{maar}
|u(x)-u(y)|\le |x-y|\left(g(x)+g(y)\right)
\end{equation}
whenever $x,y\in \Omega\setminus E.$ 

One has
$M^{1,p}(\rr^n)= W^{1,p}(\rr^n)$ as sets for $1<p\le \infty,$ and the norms
are comparable once $M^{1,p}(\rn)$ is equipped with the natural norm. Also,
for $1\le p\le\fz$, one always has 
$M^{1,p}(\Omega)\subset W^{1,p}(\Omega)$ and the inclusion
is strict for $p=1$ for any domain $\Omega$, see \cite{ks}.

A natural question to ask is:
\begin{center}
\textit{For which domains $\boz\subset\rn$ do we have $M^{1,p}(\boz)=W^{1,p}(\boz)$?}
\end{center}

 
Indeed, these two spaces coincide if 
there is a bounded extension operator from 
$W^{1,p}(\Omega)$ into $W^{1,p}(\rn),$ for a given $1<p\le \fz$.  When
$p=\fz$ and $\Omega$ is bounded, this is the case if $\Omega$ is quasiconvex and actually the equality is equivalent to quasiconvexity under these assumptions. This follows from
 \cite[Theorem 7]{hkt}. Moreover, for $1<p<\infty$, under the assumption that
\begin{equation}\label{eq:measdens}
|B(x,r)|\le C|B(x,r)\cap \Omega|
\end{equation}
for every $x\in \Omega$ and every $0<r<1,$ where $|\cdot|$ refers to $n$-measure, $M^{1,p}(\boz)=W^{1,p}(\boz)$ implies the existence of such an extension operator.
 Indeed, in this case the spaces coincide precisely when such an extension
operator exists. For this see \cite{hkt}. Using this fact, it is easy to exhibit
domains $\Omega$ for which $M^{1,p}(\Omega)= W^{1,p}(\Omega)$ 
fails for all $p;$ e.g. take $\Omega\subset\rr^2$ to be the unit disk minus the interval
$[0,1)$ on the real axis. 


In this paper, we consider this question for cuspidal domains of the form
\begin{equation}\label{cusp}
\boz_\psi:=\lf\{(t, x)\in (0, 1)\times\rr^{n-1}\setcolon |x|<\psi(t) \r\}\cup\{(t, x)\in[1,2)\times\rr^{n-1}\setcolon |x|<\psi(1)\},
\end{equation}
where $\psi\colon (0,1]\to (0,\fz)$ is a left continuous increasing function. (Left continuity
is required just to get $\boz_\psi$ open. The term ``increasing'' is used in the non-strict sense.) The seemingly strange cylindrical annexes are included only  to exclude other singularities
than the cuspidal one. 
It is crucial to note that these domains will not, except for limited special cases, be Sobolev extension domains, and thus the methods from \cite{hkt} do not apply.


\begin{figure}[h!tbp]
\centering
\includegraphics[width=0.5\textwidth]
{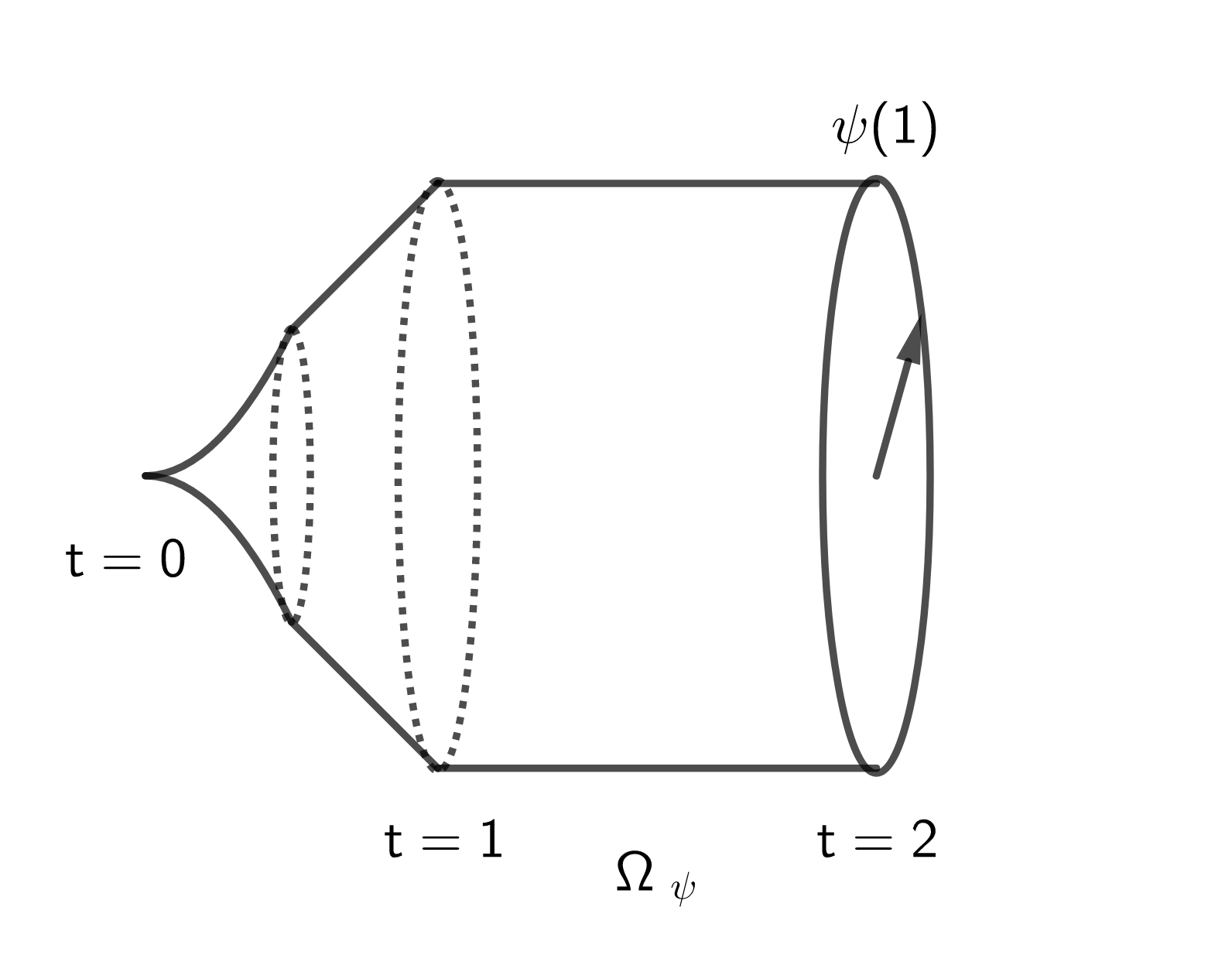}
\label{fig:cusps}
\end{figure}

It is easy to check that $\boz_\psi\subset\rn$ is a domain.  If $\lim_{t\to 0}\frac {\psi(t)}{t}=0,$ then the measure density condition \eqref{eq:measdens} fails, and hence, by \cite{hkt}, there can not exist any bounded extension operator from $W^{1, p}(\boz_\psi)$ to $W^{1,p}(\rn)$. However, according to a somewhat surprising result by A.S. Romanov \cite{Romanov}, one still has
$W^{1,p}(\boz_\psi)=M^{1,p}(\boz_\psi)$ if $\psi(t)=t^s$ with
 $s>1$ and $p>\frac{1+(n-1)s}{n}.$ Actually, Romanov proved this statement for a domain which is bi-Lipschitz equivalent to $\boz_\psi$ when $\psi(t)=t^s,$ but bi-Lipschitz transforms preserve both Sobolev and Haj\l asz-Sobolev spaces.
 
We show that the above restriction on $p$ is superfluous and 
that $\psi$ being of the form $\psi(t)=t^s$ can be relaxed to being any
 left continuous
 increasing function. 

\begin{theorem}\label{thm:main}
Let $\psi:(0, 1]\to(0, \fz)$ be a left continuous increasing function. 
Define the corresponding cuspidal domain $\boz_\psi$ as in \eqref{cusp}. Then $W^{1,p}(\boz_\psi)=M^{1,p}(\boz_\psi)$ for all $1<p\leq \infty$ with equivalence of norms.
\end{theorem}

As a consequence of the bi-Lipschitz invariance stated above,  the conclusion $M^{1,p}(\Omega)= W^{1,p}(\Omega)$ then holds
for all bi-Lipschitz images of $\Omega_{\psi}$. Thus, our result covers the result obtained by Romanov.

\section{Definitions and Preliminaries}

In what follows, $\boz\subset\rn$ is always a domain. We write 
\begin{equation}
\rn=\rr\times\rr^{n-1}:=\{z:=(t, x)\in\rr\times\rr^{n-1}\}\, .\nonumber
\end{equation}
Throughout the paper, we consider
a left continuous increasing function $\psi\colon (0,1]\to (0,\fz)$,
extend the definition of $\psi$ to the interval $(0, 2)$ by setting 
\begin{equation}
\psi(t)=\psi(1), \ \ {\rm for\ every}\ \ t\in (1,2) \nonumber
\end{equation}
and write
\begin{equation}
\boz_\psi=\{(t, x)\in (0, 2)\times\rr^{n-1}\setcolon |x|<\psi(t)\}\, .\nonumber
\end{equation}

Typically, $c$ or $C$ 
will be constants that depend on various parameters and may differ even on the same line of inequalities. 
The Euclidean distance between points $x,y$ in the Euclidean space $\rn$ is denoted by $|x-y|$. The open 
$m$-dimensional
ball of radius $r$ centered at the point $x$ is denoted by $B^{m}(x,r)$.

The space of locally integrable functions is denoted by $L^1_{\rm loc}(\boz)$. 
For every measurable set $Q\subset\rr^n$ with $0<|Q|<\infty$, and every 
non-negative measurable or integrable function $f$ on $Q$ we define the integral average of $f$ over $Q$ by 
\begin{equation}
\vint_{Q}f(w)\, dw:=\frac{1}{|Q|}\int_{Q}f(w)\, dw\, .\nonumber
\end{equation} 

 
Let us give the definitions of Sobolev space $W^{1,p}(\boz)$ and Haj\l asz-Sobolev space $M^{1,p}(\boz)$.
\begin{defn}\label{Sobolev}
We define the first order Sobolev space $W^{1,p}(\boz)$, $1\leq p\leq \fz$, as the set 
\begin{equation}
\left\{\, u\in\lp \setcolon \nabla u\in L^p(\boz\setcolon\rn)\, \right\}\, .\nonumber
\end{equation}
Here $\nabla u=\left(\frac{\partial u}{\partial x_1}\, ,\, \dots \, ,\,\frac{\partial u}{\partial x_n}\right)$ is the weak (or distributional) gradient of a locally integrable function $u$. 
\end{defn}
We equip $W^{1,p}(\boz)$ with the non-homogeneous norm:
\begin{equation}
\|u\|_{\wp}=\|u\|_{L^p(\boz)}+\||\nabla u|\|_{L^p(\boz)}\nonumber
\end{equation}
for $1\leq p<\fz$, and 
\begin{equation}
\|u\|_{W^{1,\fz}(\boz)}=\|u(z)\|_{L^\infty(\boz)}+\||\nabla u(z)|\|_{L^\infty(\Omega)}\, .\nonumber,
\end{equation}
where $\|f\|_{L^p(\boz)}$ denotes the usual $L^p$-norm for $p \in [1,\infty]$.

For $u\in\lp$, we denote by $\mathcal{D}_p(u)$ the class of functions $0\leq g\in\lp$ for which there exists $E\subset\boz$ with $|E|=0$, so that 
\begin{equation}
|u(z_1)-u(z_2)|\leq |z_1-z_2|\left(g(z_1)+g(z_2)\right), \ \ {\rm for} \ \ z_1, z_2\in\boz\setminus E\, .\nonumber
\end{equation}
\begin{defn}\label{hajlasz}
We define the Haj\l asz-Sobolev space $M^{1,p}(\boz)$, $1\leq p\leq \fz$, as the set 
\begin{equation}
\left\{u\in\lp, \mathcal{D}_p(u)\neq\emptyset\right\}\, .\nonumber
\end{equation}
\end{defn}
We equip $M^{1,p}(\boz)$ with the non-homogeneous norm:
\begin{equation}
\|u\|_{\mp}=\|u\|_{\lp}+\inf_{g\in\mathcal{D}_p(u)}\|g\|_{\lp}\, .\nonumber
\end{equation}
for $1\leq p<\fz$, and 
\begin{equation}
\|u\|_{M^{1,\fz}(\boz)}=\|u(z)\|_{L^\infty(\boz)}+\inf_{g\in\mathcal{D}_p(u)} \| g(z) \|_{L^\infty(\boz)}\, .\nonumber
\end{equation}

\section{Maximal functions}


We will define two maximal functions. The first, $M^{\tau}[f]$, will vary only the first component $t$, and the second $M^{\chi}[f]$ will vary the $x$-component. 
For every $x\in B^{n-1}(0,\psi(1))$ set 
\begin{equation}
S_x:=\{t\in\rr\setcolon (t,x)\in
\Omega_\psi
\} .\nonumber
\end{equation}
Let
 $f\colon\Omega_{\psi}\to\rr$ be
measurable and let $(t, x)\in \Omega_\psi$. We define the one-dimensional maximal function in the direction of the first variable by setting
\begin{equation}\label{t-max fun}
M^{\tau}[f](t, x):=\sup_{[a,b]\ni t}\vint_{[a, b]\cap S_x}|f(s,x)|\, ds\, .
\end{equation}
The supremum is taken over all intervals $[a,b]$ containing $t$. 

On the other hand, the second maximal function will be defined for functions 
$f\colon (0,2)\times \rr^{n-1}\to\rr$.
 For every point $(t,x)\in(0,2)\times \rr^{n-1}$, we define the $(n{-}1)$-dimensional maximal function $M^{\chi}[f]$ by setting
\begin{equation}\label{t-Max fun}
M^{\chi}[f](t, x):=\sup_{B^{n-1}(x', r)\ni x}\vint_{B^{n-1}(x', r) }|f(t,y)|\, dy\, ,
\end{equation}
where we take the supremum over the $(n{-}1)$-dimensional balls for 
which $x\in B^{n-1}(x', r)$. 
The 
next lemmas tell us that both $M^\tau$ and $M^{\chi}$ enjoy the 
usual $L^p$-boundedness property.
\begin{lemma}\label{lem:M}
Let $1<p<\fz$. Then for every $f\in L^p(\boz_\psi)$, $M^\tau[f]$ is measurable and we have 
\begin{equation}\label{equa1}
\int_{\Omega_\psi }\lf|M^\tau[f](z)\r|^p\, dz\leq C\int_{\Omega_\psi }\lf|f(z)\r|^p\, dz\, ,
\end{equation}
where the constant $C$ is independent of $f$. 
\end{lemma}
\begin{proof}
Since the maximal function comes out the same if we consider only segments with rational endpoints,
it preserves measurability.
Fubini's theorem implies that $f(\cdot,x)\in L^p(S_x)$ for almost every $x\in B^{n-1}(0,\psi(1))$. By the $L^p$-boundedness of the classical Hardy-Littlewood maximal function on the interval $S_x$, for such $x$ 
we have 
\begin{equation}\label{equa:Lpbound}
\int_{S_x}\lf|M^{\tau}[f](t,x)\r|^p\, dt\leq C\int_{S_x}\lf|f(t,x)\r|^p\, dt,
\end{equation}
where the constant $C$ is independent of $f$ and $x$. By combining the inequality (\ref{equa:Lpbound}) and Fubini's theorem together, we obtain 
\begin{eqnarray}
\int_{\boz_\psi}\lf|M^\tau[f](t, x)\r|^p\, dx\, dt
&=&\int_{B^{n-1}(0,\psi(1))}\int_{S_x}\lf|M^\tau[f](t, x)\r|^p\, dt\, dx\nonumber\\
&\leq&C\int_{B^{n-1}(0,\psi(1))}\int_{S_x}\lf|f(t, x)\r|^p\, dt\, dx\nonumber\\
&=&C\int_{\boz_\psi}\lf|f(t, x)\r|^p\, dx\, dt\, .\nonumber
\end{eqnarray}
\end{proof}
\begin{lemma}\label{lem:M2}
Let $1<p<\fz$. Then for every $f\in L^p((0,2)\times\rr^{n-1})$, 
$M^\chi[f]$ is measurable and we have 
\begin{equation}\label{equa2}
\int_{(0,2) \times \R^{n-1}}\lf|M^{\chi}[f](z)\r|^p\, dz\leq C\int_{(0,2) \times \R^{n-1}} \lf|f(z)\r|^p\, dz\, ,
\end{equation}
where the constant $C$ is independent of $f$. 
\end{lemma}
\begin{proof}
Again, the maximal function preserves measurability, as
it comes out the same if we consider only balls with rational centers and radii
(a point is rational if all its coordinates are rational).
By Fubini's theorem,  $f(t,\cdot)\in L^{p}(\rr^{n-1})$ 
for almost every $t \in (0,2)$. By the $L^p$-boundedness of the Hardy-Littlewood maximal operator we have
\begin{equation}
\int_{\rr^{n-1}}\lf|M^{\chi}[f](t, x)\r|^p\, dx\leq C\int_{\rr^{n-1}}\lf|f(t,x)\r|^p\, dx\, ,\nonumber
\end{equation}
where the positive constant $C$ is independent of $f$ and $t$. Then  Fubini's theorem gives 
\begin{eqnarray}
\int_{(0,2) \times \R^{n-1}}\lf|M^{\chi}[f](z)\r|^p\, dz&=&\int_0^2\int_{\rr^{n-1}}\lf|M^{\chi}[f](t, x)\r|^p\, dx\, dt\nonumber\\
 &\leq&C\int_0^2\int_{\rr^{n-1}}\lf|f(t, x)\r|^p\, dx\, dt\nonumber\\
&\leq&C\int_{(0,2) \times \R^{n-1}}\lf|f(z)\r|^p\, dz\, .\nonumber
\end{eqnarray}
\end{proof}

\section{Proof of the Main theorem}

Let us begin by sketching a simple proof for Theorem \ref{thm:main} in the Euclidean plane $\rr^2$,
for $1<p<\infty$.
In this case the maximal function $M^{\chi}[f]$, with respect to $x$-coordinate, can be replaced by 
\begin{equation}\label{equa:maxix}
\tilde M^{\chi}[f](t, x):=\sup_{[z, w]\ni x}\vint_{\{y\in [z, w]\setcolon (t,y)\in\boz_\psi\}}|f(t, y)|\, dy\, ,
\end{equation}
for every $(t, x)\in\boz_\psi$. As in Lemma \ref{lem:M} we obtain 
\begin{equation}\label{equa3}
\int_{\boz_\psi}|\tilde M^{\chi}[f](z)|^p\, dz\leq C\int_{\boz_\psi}|f(z)|^p\, dz\, .
\end{equation}
By \cite{Pitor}, there is a bounded inclusion $\iota \co M^{1,p}(\boz_\psi) \hookrightarrow W^{1,p}(\boz_\psi)$.  To show that $\iota$ is an isomorphism, it suffices to show that its inverse $\iota^{-1}$ is both densely defined and bounded on $W^{1,p}(\boz_\psi)$. Let $C^1(\boz_\psi)$ be the set of continuously differentiable functions. Since $C^1(\boz_\psi)\cap W^{1,p}(\boz_\psi)$ is dense in $W^{1,p}(\boz_\psi)$, it suffices to show that $C^1(\boz_\psi) \cap W^{1,p}(\boz_\psi) \subset M^{1, p}(\boz_\psi)$ and that for each $u \in C^1(\boz_\psi) \cap W^{1,p}(\boz_\psi)$ we have $||u||_{M^{1,p}(\Omega_\psi)} \lesssim ||u||_{W^{1,p}(\boz_\psi)}$. 

Fix $u\in C^1(\boz_\psi)\cap W^{1,p}(\boz_\psi)$. Let $z_1:=(t_1, x_1), z_2:=(t_2, x_2)\in\boz_\psi$ be arbitrary. Without loss of generality, we assume $0<t_1\leq t_2<2$. From the definition of $\boz_\psi$, the point $z':=(t_2, x_1)$ is also in $\boz_\psi$. Using the triangle inequality, we have 
\begin{equation}\label{equaa1}
|u(z_1)-u(z_2)|\leq |u(z_1)-u(z')|+|u(z')-u(z_2)|\, .
\end{equation}
Since $u\in C^1(\boz_\psi)\cap W^{1,p}(\boz_\psi)$, the fundamental theorem of calculus implies 
\begin{equation}\label{equaa2}
|u(z_1)-u(z')|\leq \int_{t_1}^{t_2}
|\nabla u(s,x_1)|ds\leq |z_1-z_2|M^\tau[|\nabla u|](z_1)
\end{equation}
and 
\begin{equation}\label{equaa3}
|u(z')-u(z_2)|\leq \int_{x_1}^{x_2}|\nabla u(t_2,y)|dy\leq |z_1-z_2|
\tilde M^{\chi}
[|\nabla u|](z_2)\, .
\end{equation}
Combining inequalities (\ref{equaa1}), (\ref{equaa2}) and (\ref{equaa3}) together, we have 
\begin{equation}
|u(z_1)-u(z_2)|\leq |z_1-z_2|\lf(M^{\tau}[|\nabla u|](z_1)+\tilde M^{\chi}[|\nabla u|](z_2)\r) \leq |z_1-z_2|(g(z_1) + g(z_2))\, ,\nonumber
\end{equation}
where
\begin{equation}
g(z):=M^\tau[|\nabla u|](z)+\tilde M^{\chi}[|\nabla u|](z)\, .\nonumber
\end{equation}
By inequalities (\ref{equa1}) and (\ref{equa3}), we have 
\begin{equation}
\int_{\boz_\psi}|g(z)|^pdz\leq C\int_{\boz_\psi}|\nabla u(z)|^p\, dz\nonumber
\end{equation}
which immediately gives that $g \in \mathcal{D}_p(u)$, and $\| u\|_{
M^{1,p}(\boz_\psi)} \leq C\| u\|_{
W^{1,p}(\boz_\psi)}$.

In higher dimensions, we have to work harder. Let us fix some notation.  


Let $\eta\colon \rr^{n-1}\to\rr$ be a smooth cut-off function such that
$\eta=1$ on $B^{n-1}(0,1)$ and $\eta=0$ on the complement of $B^{n-1}(0,2)$.
Consider the standard extension operator $E^R\colon W^{1,p}(B^{n-1}(0,R))\to  W^{1,p}(\rr^{n-1})$
given by
$$
E^Ru(x)=
\begin{cases}
u(x),&|x|<R,\\
0,&|x|=R,\\
u\big(\frac{R^2}{|x|^2}\,x\big)\eta\big(\frac{x}{R}\big),&|x|>R.
\end{cases}
$$
Then
\begin{equation}\label{extension}
\|\nabla E^Ru\|_{L^p(\rr^{n-1})}\le C\|\nabla u\|_{L^p(B^{n-1}(0,R))}
\end{equation}
with $C$ independent of $u$ and $R$.

Let $u\in W^{1,p}(\boz_\psi)$ be arbitrary, $1<p<\fz$. Extend the function $u$ to $(0,2)\times\rr^{n-1}$
by setting 
\begin{equation}\label{howtildeu}
\tilde u(t,\cdot)=E^{\psi(t)}(u(t,\cdot)),\qquad t\in(0,2).
\end{equation}

Denoting the gradient with respect to the $x$-variable by $\nabla^{\chi}$, 
from \eqref{af} we immediately obtain
\begin{equation}\label{eq:est0}
|\tilde u(z_1)-\tilde u(z_2)|\leq C|z_1-z_2|(M^{\chi}[|\nabla^{\chi} \tilde u|](z_1)+M^\chi[|\nabla^{\chi}\tilde u|](z_2))\, 
\end{equation}
for a.e.\ $t\in (0,2)$ and a.e.\ $z_1$, $z_2\in \{t\}\times \rr^{n-1}$. It is easily seen, when $u \in C^1(\Omega_\psi)$, that the function $\tilde u$ and $\nabla^{\chi} \tilde u$ are measurable on $(0,2)\times\rr^{n-1}$. In fact, it could be shown that both of these would be measurable even if  $u$ were just in $W^{1,p}(\Omega_\psi)$.

Next, we prove the main estimate.
\begin{lemma}\label{l:est1}
Let $z_1=(t_1, x_1), z_2:=(t_2, x_2)\in\boz_\psi$ be two points with $t_1<t_2$.  Suppose that $u \in W^{1,p}(\Omega_{\psi})\cap C^1(\boz_\psi)$ and that $\tilde u$ is its extension given by
\eqref{howtildeu}.
Then we have
\begin{eqnarray}\label{eq:est1}
|u(z_1)- u(z_2)| &\leq& C|z_1-z_2|\big(M^{\tau}[|
\nabla u
|](z_1) \ +\  M^{\tau}[M^\chi[|\nabla^\chi \tilde u|]](z_1) \ +\ \nonumber \\
&&\ \ \ \ \ \ \ \ \  \ \ \ \ \,  M^{\tau}[|
\nabla u
|](z_2) \ +\  M^{\tau}[M^\chi[|\nabla^\chi \tilde u|]](z_2)\big)\,.
\end{eqnarray}
\end{lemma}

\begin{proof}
Similarly to the two-dimensional argument, we will compare the change in the function via additional values $\tilde u(s,x_i)$ for some $s\in(0,2)$. Without knowing exactly which $s$ yields an optimal estimate, we will instead average over a range of possible $s$ with the hope that, on average, the differences are better controlled. Indeed, let 
$$
\aligned
T_2&=\min\Big\{2,t_2 + \frac{t_2-t_1}{2}\Big\},\\
T_1&=T_2-\frac{t_2-t_1}{2}.
\endaligned
$$
Notice that $t_2\in[T_1,T_2]$ and $[T_1,T_2]\times \{x_1,x_2\}\subset\Omega_{\psi}$.
When we average over different possible $s \in [T_1, T_2]$ and use the triangle inequality we obtain that
\begin{eqnarray}
\lf|u(z_2)-u(z_1)\r|\leq\underbrace{\lf| \frac{1}{T_2-T_1}\int_{T_1}^{T_2} |u(t_2,x_2)- u(s,x_2)| \, \, ds \r|}_{I} \!\!&+& \!\! \underbrace{ \lf| \frac{1}{T_2-T_1}\int_{T_1}^{T_2}|u(s,x_2) -  u(s,x_1)| \, \, ds \r|}_{II} \nonumber\\
 &+&\!\! \underbrace{\lf| \frac{1}{T_2-T_1}\int_{T_1}^{T_2} |u(s,x_1)-  u(t_1,x_1)| \, \, ds\r|}_{III}\, .\label{equa:est2}
\end{eqnarray}

First, we estimate the terms $I$ and $III$. Let $i\in\{1,2\}$.
If $t_i<s$, by the fundamental theorem of calculus we have
\begin{equation}\label{interior}
|u(t_i,x_i)- u(s,x_i)|\le \int_{t_i}^s|\nabla u(r,x_i)|\,dr\le |t_i-s|M^{\tau}[|\nabla u|](z_i)
\le 3(T_2-T_1) M^{\tau}[|\nabla u|](z_i).
\end{equation}
Similarly, \eqref{interior} holds also if $t_i\ge s$.
Integrating with respect to $s$ we obtain
\begin{equation}\label{eq:1}
I\le 3(T_2-T_1) M^{\tau}[|\nabla u|](
z_2
)\le 2|z_2-z_1| M^{\tau}[|\nabla u|](
z_2
).
\end{equation}
and
\begin{equation}\label{eq:2}
III\le 3(T_2-T_1) M^{\tau}[|\nabla u|](
 z_1
 )\le 2|z_2-z_1| M^{\tau}[|\nabla u|](
 z_1)
\end{equation}
Next, we apply \eqref{eq:est0} to the second term:
\begin{eqnarray}\label{eq:3}
II &\leq&  \frac{C|x_1-x_2|}{T_2-T_1}\int_{T_1}^{T_2} 
(M^{\chi}[|\nabla^\chi \tilde u|](s,x_1)+M^\chi[|\nabla^\chi \tilde u|](s,x_2)) \, \, ds \nonumber \\
&\leq& 
  C|x_1-x_2|\left(\frac{1}{T_2-t_1}\int_{t_1}^{T_2} (M^{\chi}[|\nabla^\chi \tilde u|](s,x_1) \, \, ds +  \frac{1}{T_2-T_1}\int_{T_1}^{T_2} (M^{\chi}[|\nabla^\chi \tilde u|](s,x_2) \, \, ds \right)
 \nonumber 
\\
&\leq&  C|z_1-z_2|\big( M^\tau [M^\chi[|\nabla^\chi \tilde u|]](z_1) + M^\tau [M^\chi[|\nabla^\chi \tilde u|]](z_2) \big) \, .
\end{eqnarray}

Finally, by combining inequalities (\ref{eq:1}), (\ref{eq:2}), (\ref{eq:3}) and (\ref{equa:est2}), we obtain the desired inequality (\ref{eq:est1}).
\end{proof}

Recall that a domain $\Omega$ is quasiconvex if there exists a $C \geq 1$ such that, for every pair of points $x,y \in \Omega$, there is a rectifiable curve $\gamma \subset \Omega$ joining $x$ to $y$ so that ${\rm len}(\gamma) \leq C|x-y|$.

\begin{proof}[Proof of Theorem \ref{thm:main}] Because $\Omega_\psi$ is quasiconvex for every $\psi$, the case of $p=\infty$ is a consequence of \cite[Theorem 7]{hkt}. Thus, fix $1<p<\fz$. By \cite{Pitor}, we know that there is a bounded inclusion $\iota \colon M^{1,p}(\boz_\psi) \hookrightarrow W^{1,p}(\boz_\psi)$.  To show that $\iota$ is an isomorphism it suffices to show that the dense subspace $C^1(\boz_\psi)\cap W^{1,p}(\boz_\psi)$ of $W^{1,p}(\boz_\psi)$ is contained in $M^{1,p}(\boz_\psi)$, and that the restricted inverse $\iota^{-1}|_{C^1(\boz_\psi)\cap W^{1,p}(\boz_\psi)}$ is defined and bounded.

Let $u\in C^{1}(\boz_\psi)\cap W^{1,p}(\boz_\psi)$ be arbitrary, and define $\tilde u$ 
as in \eqref{howtildeu}.
Set
\begin{equation}\label{equa:ud}
\hat g(z)=M^\tau[|\nabla  u|](z)+M^\chi[|\nabla^{\chi} \tilde u|](z)
+M^\tau[M^\chi[|\nabla^{\chi} \tilde u|]](z)\, .
\end{equation}

By \eqref{eq:est0} and Lemma \ref{l:est1},  for every $z_1, z_2\in\boz_\psi$, we get the estimate
\begin{equation}
|u(z_1)-u(z_2)|\leq C|z_1-z_2|(\hat g(z_1)+\hat g(z_2))\, .\nonumber
\end{equation}
Hence (\ref{maar}) holds for $g:=C\hat g$ for a suitable constant $C>1$. The triangle inequality gives
\begin{equation}
\int_{\boz_\psi}|g(z)|^p\, dz\leq C\lf(\int_{\boz_\psi}M^\tau[|\nabla  u|](z)^p\, dz+\int_{\boz_\psi}M^\chi[|\nabla^{\chi} \tilde u|](z)^p\,dz+\int_{\boz_\psi}M^\tau[M^\chi[|\nabla^\chi \tilde u|]](z)^p\, dz\r)\, .\nonumber
\end{equation}
Lemmata \ref{lem:M} and \ref{lem:M2} and \eqref{extension} lead to the estimates
 \begin{equation}
 \int_{\boz_\psi}|M^\tau[|\nabla  u|](z)|^p\, dz\leq C \int_{\boz_\psi}|\nabla  u(z)|^p\, dz\nonumber
 \end{equation}
 and 
 \begin{eqnarray}
 \int_{\boz_\psi}|M^\tau[M^\chi[|\nabla^{\chi}\tilde u|]](z)|^p\, dz&\leq& C\int_{\boz_\psi}M^\chi[|\nabla^{\chi}\tilde u|](z)^p\, dz
 \leq C\int_{(0,2)\times\rr^{n-1}}|\nabla^{\chi}\tilde u(z)|^p\, dz\nonumber\\
 &\leq& C\int_0^2\int_{\rr^{n-1}}|\nabla^{\chi}\tilde u(t,x)|^p\, dx\;dt
 \leq C\int_0^2\int_{B(0,\psi(t))}|\nabla^{\chi} u(t,x)|^p\, dx\;dt\nonumber\\
 &\leq& C\int_{\boz_\psi}|\nabla u(z)|^p\, dz\, ,\nonumber
 \end{eqnarray}
 which imply that $g\in\mathcal{D}_p(u)$ and that $\|u\|_{M^{1,p}(\boz_\psi)} \leq C \|u\|_{W^{1,p}(\boz_\psi)}$. That is, $\iota^{-1}|_{C^1(\boz_\psi)\cap W^{1,p}(\boz_\psi)}$ is both well-defined and bounded.
\end{proof}


\begin{thebibliography}{99}
\bibitem{AandF}
E.~Acerbi and N.~Fusco, \textit{An approximation lemma for $W^{1,p}$-functions}. Material instabilities in continuum mechanics (Edinburgh, 1985–1986), 1–5, Oxford Sci. Publ., Oxford Univ. Press, New York, 1988.


\bibitem{BandH}
B.~Bojarski and P.~Haj\l asz, 
\textit{Pointwise inequalities for Sobolev functions and some applications}, 
Studia Math. 106 (1993), no. 1, 77--92.

\bibitem{Evens}
L.~C.~Evans and R.~F.~Gariepy,
\textit{Measure Theory and Fine Properties of Functions},
Revised ed., in: Textbooks in Mathematics, CRC Press, Boca Raton, FL, 2015.


\bibitem{Pitor}
P.~Haj\l asz,
\textit{Sobolev spaces on an arbitrary metric space},
Potential Anal.~5(1996), no. 4, 403--415.


\bibitem{hkt}
P.~Haj\l asz, P.~Koskela and H.~Tuominen,  
\textit{Sobolev embeddings, extensions and measure density condition},
J.~Funct.~Anal.~254 (2008), no.~5, 1217--1234. 




\bibitem{HKST}
J.~Heinonen, P.~Koskela, N.~Shanmugalingam and J.~Tyson, \textit{Sobolev spaces on metric measure spaces. An approach based on upper gradients}, New Mathematical Monographs, 27. Cambridge University Press, Cambridge, 2015.



\bibitem{ks}
P.~Koskela and E.~Saksman
\textit{Pointwise characterizations of Hardy-Sobolev functions}
Math.~Res.~Lett.~15 (2008), 727--744.




\bibitem{Lewis}
J.~Lewis, \textit{On very weak solutions of certain elliptic systems}, Comm.~Partial Differential Equations 18 (1993), no. 9-10, 1515--1537.


\bibitem{Romanov}
A.~S.~Romanov,
\textit{On a generalization of Sobolev spaces},
Sibirsk.~Mat.~Zh.~39 (1998), no.~4, 949-953;
traslation in Siberian Math.~J.~39 (1998), no.~4, 821--824.


\bibitem{Stein}
E.~M.~Stein, \textit{Harmonic analysis: real-variable methods, orthogonality, and oscillatory integrals}, With the assistance of Timothy S. Murphy, Princeton Mathematical Series, 43. Monographs in Harmonic Analysis, III. Princeton University Press, Princeton, NJ, 1993.
\end{thebibliography}
\end{document}